\documentclass[11pt]{amsart}

\usepackage{amsmath,amssymb,amsbsy,amsthm,latexsym,
         amsopn,amstext,amsxtra,euscript,amscd,amsthm, mathtools, dsfont}
         
\usepackage{fullpage,amssymb,amsmath}
\usepackage{enumerate}
\usepackage[shortlabels]{enumitem} 
\usepackage{comment}
\usepackage{hyperref} 
\usepackage[capitalise,nameinlink]{cleveref}
\usepackage{url}
\usepackage{xcolor}
\usepackage{amsmath,amsthm,amssymb, graphicx, multicol, array, mathtools, enumerate,enumitem,hyperref,algorithm,algorithmic,tikz}
\usepackage{fancyhdr}
\usepackage{upgreek}
\usepackage{physics}
\usepackage{xcolor}

\usepackage{bbm}
\usepackage{mathrsfs} 
\newtheorem*{corollary*}{Corollary}
\newtheorem*{theorem*}{Theorem}
\newtheorem{theorem}{Theorem} 
\newtheorem{corollary}[theorem]{Corollary}
\newtheorem{proposition}[theorem]{Proposition}
\newtheorem{lemma}[theorem]{Lemma}

\newtheorem*{question*}{Question}

\theoremstyle{definition}

\theoremstyle{remark}

\numberwithin{equation}{section}

\newcommand{\N}{\mathbb{N}}

\newcommand{\R}{\mathbb{R}}

\renewcommand{\epsilon}{\varepsilon}

\usepackage{pgfplots}
\pgfplotsset{compat=1.14}
\pgfplotsset{every tick label/.append style={font=\tiny}}
\usepgfplotslibrary{fillbetween}
         
\begin{document}

\title{Failure of $L^p$ Symmetry of Zonal Spherical Harmonics}

\author{Gabriel Beiner}
\address{Department of Mathematics\\ University of Toronto\\ 40 St. George Street\\ Toronto\\
ON\\ Canada\\ M5S 2E4
} 
\email{gabriel.beiner@mail.utoronto.ca}

\author{William Verreault}
\address{
D\'{e}partement de Math\'{e}matiques et de Statistique\\ 
Universit\'{e} Laval\\ 
Qu\'ebec\\
QC\\
G1V 0A6 \\
Canada} 
\email{william.verreault.2@ulaval.ca}

\maketitle

\vspace*{-4mm}

\begin{abstract}
In this paper, we show that the 2-sphere does not exhibit symmetry of $L^p$ norms of eigenfunctions of the Laplacian for $p\geq 6$. In other words, there exists a sequence of spherical eigenfunctions $\psi_n$, with eigenvalues $\lambda_n\to\infty$ as $n\to\infty$, such that the ratio of the $L^p$ norms of the positive and negative parts of the eigenfunctions
does not tend to $1$ as $n\to\infty$ when $p\geq 6$. Our proof relies on fundamental properties of the Legendre polynomials and Bessel functions of the first kind.

\end{abstract}

\section{Introduction} \label{sec:intro} 

The statistical properties of eigenfunctions of the Laplace--Beltrami operator on general Riemannian manifolds has been a fruitful area of reasearch. One area of interest, based on conjectures of quantum chaos \cite{Berry_Conjec, Marklof}, is the study of symmetries of the positive and negative parts of these eigenfunctions. Jakobson and Nadirashvili \cite{J-N} have in particular investigated the ratio of their $L^p$ norms, 
proving the following result.
\begin{theorem}[Jakobson--Nadirashvilli, \cite{J-N}] \label{thm:j-n}
Let $M$ be a smooth compact manifold and $p \geq 1$. Then there exists $C>0$, depending only on $p$ and the manifold $M$, such that for any nonconstant eigenfunction $\psi$ of the Laplacian,
      \[1/C \leq \|\psi_+\|_p/\|\psi_-\|_p \leq C.\]
\end{theorem}
Here, $\psi_+$ and $\psi_-$ stand for the positive and negative parts of $\psi$, respectively. Analogous quasi-symmetry results for the support of the volume of these positive and negative parts were first obtained by Donnelly and Fefferman \cite{MR943927} while they were investigating symmetry distribution problems in relation with Yau's conjecture.
At the end of their paper, Jakobson and Nadirashvili ask whether this ratio always tends to one as the corresponding eigenvalue goes to infinity on a given manifold for $p> 1$ (also see \cite{jakobson2001geometric}). 
They comment that the even zonal spherical harmonics on the 2-sphere provide a case where this fails for the $L^\infty$ norm. In this paper, we extend that result on the 2-sphere to all $p\geq 6$.

\begin{theorem} \label{thm:main}
For $p\geq 6$, there exists a sequence of eigenfunctions $\psi_n$ on the 2-sphere, with eigenvalues $\lambda_n\to\infty$ as $n\to\infty$, such that
\[\lim_{n\to\infty}\frac{\|\psi_{n,+}\|_p}{\|\psi_{n,-}\|_p} > 1 .\]

\end{theorem}

The authors 
along with Eagles and Wang \cite{FUSRP1} have already shown a case where symmetry fails on the standard flat $d$-torus for $d\geq 3$. Their argument relies on an example of Mart\'inez and Torres de Lizaur \cite{eigentorus} used to disprove the symmetry conjecture on the distribution of the eigenfunctions, that is, to show that the ratio of the volume of the support of $\psi_+$ to $\psi_-$ in the high-energy limit does not tend to $1$. The proof on the torus involves computational methods and uses the symmetry of the torus to generate a sequence from rescaling a single eigenfunction. In contrast, our proof relies purely on classical results about orthogonal polynomials and features a bona fide sequence of distinct eigenfunctions.

The argument in this paper also supplants the previous work on this question in a number of ways. Mart\'inez and Torres de Lizaur \cite{eigentorus} have shown that in the case of the flat 2-torus, $L^p$ symmetry holds for every eigenfunction, and so our argument provides the first case of the failure of symmetry for a 2-dimensional manifold and for a non-flat manifold. Mart\'inez and Torres de Lizaur \cite{signlegendre} have also shown that for the even spherical harmonics we use in our proof, the distribution ratio of the volume of supports of $\psi_+$ to $\psi_-$ tends to one as the corresponding eigenvalue tends to infinity. As such, our result is the first example of a sequence of eigenfunctions which have asymptotic distribution symmetry but not $L^p$ norm symmetry in the high energy limit. Lastly, since our result is primarily one about the Legendre polynomials, it may also be of interest to those studying the asymptotic behaviour of orthogonal polynomials independent of any of the geometric motivations underlying our study.


\section{Preliminaries and notation} 


For $n\in \N$, we denote the $n$th Legendre polynomial by $P_n$. We restrict our attention to $P_n$ in the domain $[0,1]$. Legendre polynomials can be defined in several equivalent ways (see \cite{pollard} pp.\:591-605). We will use their differential equation definition: for $x\in (-1,1)$,
\begin{equation}\label{eq:odeleg}
        P''_n(x) =\frac{1}{1-x^2}(2xP'_n(x) -n(n+1)P_n(x)),
    \end{equation} 
with the inital condition $P_n(1)=1$. From this definition we can recover $P'_n(1)=n(n+1)/2$. We will also use the fact that the $P_n$ satisfy Bonnet's recursion formula
\begin{equation} \label{eq:bonnet}
    (n+1)P_{n+1}(x)=(2n+1)xP_n(x)-nP_{n-1}(x),
    \end{equation}
and that Bernstein's inequality (\cite{szego_orth_pol}, p.\:165) gives a classic bound on $P_n(x)$ for $n\in\N$ and $x\in(-1,1)$:
\begin{equation} \label{eq:bernstein}
    |P_n(x)| \leq \sqrt{\frac{2}{\pi n}}(1-x^2)^{-1/4}.
\end{equation}

We label the positive zeroes of $P_n$ as $z_{i,n}$ for $i\in\{1, 2, \ldots , \lfloor n/2\rfloor\}$, where $0< z_{\lfloor n/2\rfloor,n} <z_{\lfloor n/2\rfloor-1,n} <\ldots < z_{1,n} <1$. Sometimes we abbreviate $z_{1,n}$ as $z_n$. A classic result of Bruns \cite{szego_bounds} gives estimates for $z_{i,n}$:
    \begin{equation} \label{eq:bruns}\cos\Big(\frac{i-\frac{1}{2}}{n+\frac{1}{2}}\pi\Big) \leq z_{i,n} \leq \cos\Big(\frac{i}{n+\frac{1}{2}}\pi\Big).
    \end{equation}
  
We label the local extremal points of $P_n$ as $x_{i,n}$ and their corresponding absolute values $|P_n(x_{i,n})|= y_{i,n}$ for $i\in\{1,\ldots , \lfloor (n-1)/2\rfloor\}$, where $0< x_{\lfloor (n-1)/2\rfloor,n} <x_{\lfloor (n-1)n/2\rfloor-1,n} <\ldots < x_{1,n} <1$.

We will also need to make use of the Bessel functions of the first kind, which we denote by $J_n$. We let $j_i$ denote the $i$th zero of $J_1$ greater than zero. Equivalently, since $J_0' =J_1$, $j_i$ are the critical points of $J_0$ and $J_0(j_i)$ are the local extrema. 

Watson's classic tome on Bessel functions \cite{watson_bessel} provides a full analysis of the zeroes of these functions. It can be inferred from this analysis that 
\begin{equation}
\label{eq:watson}
    \Big(i+\frac{1}{2}\Big)\pi > j_{i} > i\pi.
\end{equation}

Indeed, Watson (pp.\:478-479) shows that all the zeroes of $J_0(x)$ lie in intervals of the form $\qty(\frac{2n-1}{2}\pi, n\pi)$ and each such interval contains at least one zero. Similarly, all the zeroes of $J_1(x)$ lie in intervals of the form $\qty(n\pi, \frac{2n+1}{2}\pi)$ and each such interval has at least one zero. Watson also proves that the zeroes of $J_0$ and $J_1$ are interlacing (pp.\:479-480), and since the intervals above do not overlap, there must be exactly one zero of $J_0, J_1$ in each interval of the above forms, respectively, from which \eqref{eq:watson} follows.

These Bessel functions will appear in our argument via the following connection to the Legendre polynomials as shown by  Cooper \cite{cooper_1950}: 
\begin{equation} \label{eq:cooper}
\lim_{n\to\infty} y_{i,n} = J_0(j_i).
\end{equation}

Finally,
Szeg{\H o} (\cite{szego_orth_pol} p.\:167) showed that for $\nu \in \qty[-\frac{1}{2},\frac{1}{2}]$, 
\begin{equation} \label{eq:szegop.167}
    |J_\nu(x)| \leq \sqrt{\frac{2}{\pi x}}.
\end{equation}

We end this subsection by proving a simple lemma about Legendre polynomials, which will be needed in the next section. We remind the reader that we abbreviate $z_{1,n}$ as $z_n$.

\begin{lemma} \label{lem:pn<x}
For $n\geq 1$ and $x\in [z_{n}, 1]$,
\[P_n(x) \leq x.\]
\end{lemma}

\begin{proof}
We proceed by strong induction, noting that the result clearly holds for $P_1(x)=x$ and $P_2(x)=\frac{1}{2}(3x^2-1)$. 
Under the assumption $x\in[z_n, 1]$, we have $0\leq x\leq 1$ and $P_n(x)\geq0$, so from Bonnet's formula \eqref{eq:bonnet}, we obtain
\begin{align*}
(n+1)P_{n+1}(x)-(n+1)P_{n}(x)&=[(2n+1)x-(n+1)]P_n(x)-nP_{n-1}(x) \\
&\leq nP_n(x)-nP_{n-1}(x),
\end{align*}
hence
$$
P_{n-1}(x)-P_n(x) \leq \frac{(n+1)}{n}(P_n(x)-P_{n+1}(x)).
$$
Then as long as $P_{n-1}\geq P_n$ on $[z_n, 1]$, we have $P_{n}\geq P_{n+1}$ on $[z_{n+1},1]\subseteq [z_n, 1]$. Since $P_1\geq P_2$ on $[0,1]$, it follows by strong induction that $x=P_1  \geq P_2 \geq  \ldots \geq P_n$ on $[z_n,1]$ for all $n\geq 1$.
\end{proof}

\section{Failure of asymptotic symmetry of $L^p$ norms on the sphere} \label{sec:main}
We start by outlining the idea behind the proof of \cref{thm:main}. We are looking for
a lower bound on $\|\psi_{n,+}\|_p$ and
an upper bound on $\|\psi_{n,-}\|_p$, where $\psi_n$ is the $n$th even zonal spherical harmonic (often denoted by $Y_{2n}^0(\theta,\varphi)$). Up to normalization, $\psi_n = P_{2n}(\cos\theta)$ where $\theta$ is the latitude on $S^2$, and so by a change of variables, it is enough to bound the ratio of the $L^p$ norms of the even Legendre polynomials on $[0,1]$. In particular we use the subsequence of the $4n$th polynomials (this is solely for some simplifications of the algebra in the proof of Lemma 5). In what follows, the absolute value of the positive and negative parts of $P_n$ are labelled as $P_{n,+}$ and $P_{n,-}$, respectively. 
We will underapproximate the $L^p$ norm of $P_{4n,+}$ as the area of a triangle bounding from below the connected component of the support of $P_{4n,+}$ containing 1 (since this component dominates the norm for large $p$ in the semi-classical limit). We will also overapproximate the $L^p$ norm of $ P_{4n,-}$ via an upper Darboux sum, using the zeroes of $P_{4n}$ as a partition. An illustration of this approximation is shown in \cref{fig:bounds}. Precisely, we will prove the following lemmas.
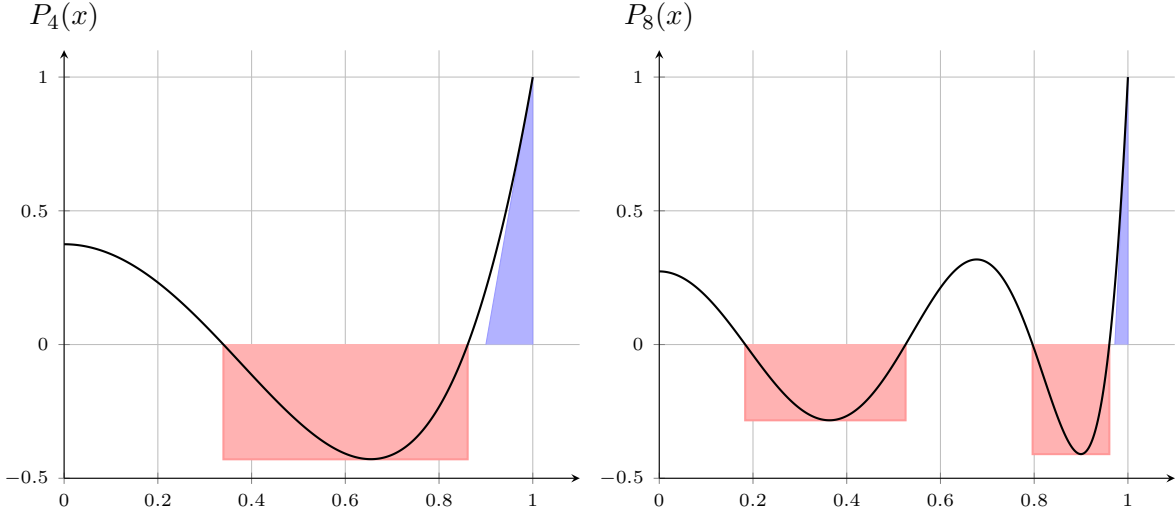
\begin{figure}
\begin{tikzpicture} 
    \begin{axis}[ylabel= $P_4(x)$,ylabel style={rotate=-90}, xmin=0, xmax =1.1, ymin=-0.5,grid, ymax=1.1, samples=500, axis lines=left, every axis y label/.style={
    at={(ticklabel* cs:1.02)},
    anchor=south,
},]
    
        \addplot [color=blue!40,fill=blue!30, 
                    ]coordinates {
            (0.9, 0) 
            (1, 1)
            (1, 0)  };
 \addplot [thick,color=red!40,fill=red!30, 
		]coordinates {
            (0.34, 0) 
            (0.34, -0.429)
            (0.861, -0.429) 
	   (0.861,0)
 };
\addplot[thick, domain=0:1, ] {(35*x^4-30*x^2+3)/8};

    \end{axis}
\end{tikzpicture}  
\begin{tikzpicture} 
    \begin{axis}[ylabel = $P_8(x)$, ylabel style={rotate=-90}, xmin=0, xmax =1.1, ymin=-0.5,grid, ymax=1.1, samples=500, axis lines=left, every axis y label/.style={
    at={(ticklabel* cs:1.02)},
    anchor=south,
},]
    
        \addplot [color=blue!40,fill=blue!30, 
                    ]coordinates {
            (0.9722, 0) 
            (1, 1)
            (1, 0)  };
 \addplot [thick,color=red!40,fill=red!30, 
		]coordinates {
            (0.1834, 0) 
            (0.1834, -0.2832)
            (0.5255, -0.2832) 
	   (0.5255,0)
 };
 \addplot [thick,color=red!40,fill=red!30, 
		]coordinates {
            (0.7967, 0) 
            (0.7967, -0.4097)
            (0.9603, -0.4097) 
	   (0.9603,0)
 };
\addplot[thick, domain=0:1, ] {(6435*x^8-12012*x^6+6930*x^4-1260*x^2+35)/128};

    \end{axis}
\end{tikzpicture}
\caption{Plots of the Legendre polynomials $P_4$ and $P_8$ indicating the corresponding approximations on the positive and negative parts of the two functions. The negative $L^1$ norm squared is overapproximated as a sum of areas of rectangles shown in red and the positive $L^1$ norm squared is underapproximated by the area of a triangle shown in blue. Higher $L^p$ norms are bounded by taking the $p$th power of the constant and linear functions corresponding to the approximation.}
\label{fig:bounds}
\end{figure}  

\begin{lemma} \label{lem:tri}
For $n\in \N$ and $p\in (0,\infty)$,
\[\int_{z_{n}}^{1} P_{n}^p  \geq \frac{2}{(p+1)n(n+1)}.\]
\end{lemma}

\begin{lemma} \label{lem:P-up}
For $n\in \N$ and $p\in (0,\infty)$,
\begin{equation} \label{eq:P-up}
    \int_0^1 P_{4n,-}^p \leq  \frac{3\pi^2}{(4n+ \frac{1}{2})^2} \sum_{i=1}^{n} iy_{2i-1,4n}^p.
\end{equation}
where as above $y_{i,n}$ denotes the $i$th largest absolute value of an extremal value of $P_n$. 
\end{lemma}

In particular, \cref{lem:tri} implies that
\begin{equation} \label{eq:ls}\int_0^1P_{4n,+}^p \geq \int_{z_{4n}}^1 P_{4n,+}^p = \int_{z_{4n}}^1 P_{4n}^p \geq \frac{1}{2(p+1)n(4n+1)}. \end{equation}

Combining these two lemmas, we will prove the following proposition, which will lead us directly to \cref{thm:main} in \cref{sec:finalproof}.

\begin{proposition} \label{prop:lim}
For $p\in(4,\infty)$, there exists a sequence of increasing natural numbers $n$ such that the sequence of quotients $\int_0^1 P_{4n,+}^p/\int_0^1 P_{4n,-}^p$ is convergent and satisfies
\[\lim_{n\to\infty} \frac{\int_0^1 P_{4n,+}^p}{\int_0^1 P_{4n,-}^p} \geq \frac{1}{p+1}\frac{2}{3\pi^2}\Big(\sum_{i=1}^{\infty}i|J_0(j_{2i-1})|^p\Big)^{-1},\] 
where as above $J_0(j_k)$ is the value of the $k$th local extrema after $x=0$ of the zeroth Bessel function of the first kind.
\end{proposition}

\subsection{Proof of \cref{lem:tri}} \label{sec:lo}

Consider the following piecewise linear function defined on $[z_{n},1]$:
\[g(x) = \begin{cases}
0&\text{if}\quad x \in \left[z_n, 1-\frac{2}{n(n+1)}\right],\\
\frac{n(n+1)}{2}x - \frac{n(n+1)}{2} +1 &\text{if}\quad x \in \left[ 1-\frac{2}{n(n+1)},1\right].
\end{cases}\]
To prove the lemma, it will be enough to show that $P_{n}(x)\geq g(x)$ on $[z_n,1]$ since
$$
\int_{z_n}^1 g(x)^p\dd{x}=\int_{1-\frac{2}{n(n+1)}}^1\Big(\frac{n(n+1)}{2}x - \frac{n(n+1)}{2} +1\Big)^p\dd{x} = \frac{2}{n(n+1)}\int_0^1 x^p\dd{x} = \frac{2}{(p+1)n(n+1)}.
$$

We start by verifying that $g$ is well-defined, i.e., that $z_n \leq 1-\frac{2}{n(n+1)}$. For the sake of contradiction, we assume that $z_n > 1-\frac{2}{n(n+1)}$ and split into two cases. First, suppose $P_{n}'(z_n) \leq  n(n+1)/2$. By the mean value theorem, there must be some point $q \in (z_n,1)$ at which $P'_{n}(q) > n(n+1)/2$. Note since  $P_{n}'(1)$ and $P_{n}'(z_n)$ are less than or equal to $ n(n+1)/2$, by the extreme value theorem applied to $P'_n$, there must be some local maximum $r\in(z_n,1)$ of $P_n'$ at which $P''_n(r)=0$ and $P'_n(r)> n(n+1)/2$. The differential equation definition of the Legendre polynomials \eqref{eq:odeleg} then yields 
\[0 = \frac{1}{1-r^2}(2rP'_n(r) -n(n+1)P_n(r)) >\frac{n(n+1)}{1-r^2} (r-P_n(r)), \]
which is a contradiction with \cref{lem:pn<x}. 
On the other hand, if $P_{n}'(z_n) >  n(n+1)/2$, then \eqref{eq:odeleg} gives
\[P''_n(z_n) =\frac{2z_nP'_n(z_n)}{1-z_n^2}>0, \]
so $P'_n$ is increasing in a neighbourhood of $z_n$. Since $P'_n(1)<P'_n(z_n)$, there must be some point $r\in(z_n,1)$ at which $P''_n(r)=0$ and $P'_n(r)> n(n+1)/2$, which is a contradiction as in the first case. Hence $g$ is well-defined. 

Note that $P_n(x)\geq g(x)$ holds trivially on $\left[z_n, 1-\frac{2}{n(n+1)}\right]$ by what we have just proved, so assume there is some point $q \in \qty(1-\frac{2}{n(n+1)},1)$ at which $P_n(q) < g(q)$. By the mean value theorem, there must then be some $r \in (q,1)$ with $P_n'(r)>n(n+1)/2$, and so, by the extreme value theorem for $P'_n$, there is some point $s \in (q,1)$ with $P'_n(s)>n(n+1)/2$ and $P''_n(s)=0$. By the same logic as above using \eqref{eq:odeleg}, we arrive at a contradiction, hence there is no $q\in [z_n,1]$ such that $P_n(q) < g(q)$.

\subsection{Proof of \cref{lem:P-up}} \label{sec:up} 
We can upper bound the integral of $P_{4n,-}^p$ using an upper Darboux sum  with a partition $\mathcal{P}_{4n}$ given by the zeroes of $P_{4n}$, that is, $\mathcal{P}_{4n}= (0, z_{2n,4n},z_{2n-1,4n}, \ldots , z_{1,4n},1)$. Since there is one local extremum between each zero and the extrema oscillate in sign, we have
\begin{equation*} \label{eq:us}\int_0^1 P_{4n,-}^p \leq  \sum_{i=1}^n (z_{2i-1,4n}-z_{2i,4n})y_{2i-1,4n}^p. \end{equation*}
Using Bruns' estimates \eqref{eq:bruns} as well as the identity $\cos(x)-\cos(y)=2\sin(\frac{x+y}{2})\sin(\frac{y-x}{2})$, we obtain
\begin{align} 
   \int_0^1 P_{4n,-}^p &\leq \sum_{i=1}^{n}\Big[\cos\Big(\frac{2i-\frac{3}{2}}{4n+\frac{1}{2}}\pi\Big)-\cos\Big(\frac{2i}{4n+\frac{1}{2}}\pi\Big)\Big]y_{2i-1,4n}^p \nonumber\\
    &= 2\sum_{i=1}^{n}\sin\Big(\frac{2i-\frac{3}{4}}{4n+\frac{1}{2}}\pi\Big)\sin\Big(\frac{3}{4(4n+\frac{1}{2})}\pi\Big)y_{2i-1,4n}^p. \label{eq:ratioin2}
\end{align}
The sine arguments in the above expression lie within $[0,\pi/2]$ and so the approximation $x\geq \sin x$ holds, hence \eqref{eq:ratioin2} is at most
\begin{align*}
\frac{3\pi^2}{2(4n+ \frac{1}{2})^2} \sum_{i=1}^{n}\Big(2i-\frac{3}{4}\Big)y_{2i-1,4n}^p
\leq \frac{3\pi^2}{(4n+ \frac{1}{2})^2} \sum_{i=1}^{n} iy_{2i-1,4n}^p.
\end{align*}

\subsection{Proof of \cref{prop:lim}}

Combining \eqref{eq:ls} with \eqref{eq:P-up}, we get
\begin{align} \label{eq:lim}
    \frac{\int_0^1 P_{4n,+}^p}{\int_0^1 P_{4n,-}^p} &\geq \frac{1}{p+1}\frac{1}{6\pi^2} \frac{(4n+ \frac{1}{2})^2}{n(4n+1)} \Big(\sum_{i=1}^{n} iy_{2i-1,4n}^p\Big)^{-1} \\
    &= \frac{1}{p+1}\Big(\displaystyle\sum_{i=1}^{n}iy_{2i-1,4n}^p\Big)^{-1} \Big( \frac{2}{3\pi^2}+\mathcal{O}(n^{-1})\Big). \nonumber
\end{align}

By \cref{thm:j-n}, for any $p \in (1,\infty)$, $\int_0^1 P_{4n,+}^p/\int_0^1 P_{4n,-}^p$ belongs to a compact interval $[1/C,C]$ independent of $n$ and so it must have a convergent subsequence. Also recall from \eqref{eq:cooper} that Cooper showed $\lim_{n\to\infty} y_{i,n} = J_0(j_i)$, so \cref{prop:lim} follows upon taking limits as $n$ goes to infinity on both sides of \eqref{eq:lim}, as long as we can push the limit inside the summation sign where we consider, for each $i$, $y_{2i-1,4n}$ as an infinite sequence in $n$ which is zero for $n < i$. Note that because of a result of Szeg{\H o} \cite{szego_orth_pol} which says that $y_{i,n}$ is decreasing in $n$ for a given $i$, the sum $\sum_{i=1}^\infty i y_{2i-1,4i}^p$, which consists of the first nonzero term of every sequence $y_{2i-1,4n}$ considered above, termwise dominates $\sum_{i=1}^{n} iy_{2i-1,4n}^p$ for each $n$.
By the dominated convergence theorem applied to the counting measure on $\N$, it suffices to verify that the dominating sum converges for $p>4$.

Recall that we defined $y_{i,n}=|P_n(x_{i,n})|$, and so
$$y_{2i-1,4i} = |P_{4i}(x_{2i-1,4i})|
\leq \frac{1}{\sqrt{2\pi i}}(1-x_{2i-1,4i}^2)^{-1/4}$$ by \eqref{eq:bernstein}.
Since the zeroes of the Legendre polynomials interlace with the critical points, and the greatest zero $z_{1,n}$ is always greater than the greatest critical point $x_{1,n}$, we know that $x_{2i-1,4i}<z_{2i-1,4i}$, so we obtain, coupling it with Bruns' inequality \eqref{eq:bruns}, 
\begin{align} 
   y_{2i-1,4i} &\leq \frac{1}{\sqrt{2\pi i}}(1-z_{2i-1,4i}^2)^{-1/4} \nonumber \\
   &\leq \frac{1}{\sqrt{2\pi i}}\Big(1-\cos^2\Big(\frac{2i-1}{4i+\frac{1}{2}}\pi\Big)\Big)^{-1/4} \nonumber\\
&=\Big(2\pi i\sin\Big(\frac{2i-1}{4i+\frac{1}{2}}\pi\Big)\Big)^{-1/2}.\label{eq:iny2}
\end{align}
Using the fact that $\sin x\geq 2x/\pi$ for $x\in[0,\pi/2]$ and $0<(2i-1)/(4i+1/2)<1/2$ for $i\geq 1$, \eqref{eq:iny2} is at most
$$
\Big(2\pi i\frac{4i-2}{4i+\frac{1}{2}}\Big)^{-1/2}.
$$
Since $(4i-2)/(4i+1/2)$ is increasing and is equal to $4/9$ for $i=1$, we have
\begin{align*}
\sum_{i=1}^\infty iy_{2i-1,4i}^p &\leq \sum_{i=1}^\infty i\Big(\frac{8\pi i}{9}\Big)^{-p/2}\\
&=\Big(\frac{3}{2\sqrt{2\pi}}\Big)^p\sum_{i=1}^\infty i^{1-p/2},
\end{align*}
which converges for $p>4$.

\subsection{Proof of \cref{thm:main}} \label{sec:finalproof}
First, let us consider $p=\infty$. The argument was outlined in \cite{J-N} but we include it here for completeness. For the $2n$th zonal spherical harmonic, we know from our analysis of the extremal points and from \eqref{eq:cooper} that
\[\lim_{n\to\infty} \frac{\|P_{2n}^+\|_{\infty}}{\|P_{2n}^-\|_{\infty}} = \lim_{n\to\infty}\frac{1}{y_{1,n}} = \frac{1}{J_0(j_1)}.\] 
Cooper also states in \cite{cooper_1950} that to four significant figures, $J_0(j_1)= 0.4027$, whence 
\[\lim_{n\to\infty} \frac{\|P_{2n}^+\|_{\infty}}{\|P_{2n}^-\|_{\infty}} \geq \frac{1}{0.403} \geq 2.48,\]
which completes the proof for $p=\infty$.

We are now ready to finish the proof of \cref{thm:main} for $6 \leq  p <\infty$, starting with the remarks made at the beginning of \cref{sec:main} which imply that it will be enough to prove that 
$$\lim_{n\to\infty} \frac{\int_0^1 P_{4n,+}^p}{\int_0^1 P_{4n,-}^p} > 1. 
$$

First, observe that 
$$ \frac{1}{p+1}\frac{2}{3\pi^2}\Big(\sum_{i=1}^{\infty}i|J_0(j_{2i-1})|^p\Big)^{-1}$$
is increasing in $p$. To see this, it is enough to show that term by term, $(p+1)|J_0(j_{2i-1})|^p$ is decreasing. From calculus, for a given $0<c<1$, one can verify that $(x+1)c^x$ is decreasing for $x> 1/\log(1/c)-1$. Since $|J_0(j_1)|<0.5$, and $|J_0(j_i)|$ is a decreasing sequence in $i$, it is enough that 
$p> 1/\log(2)-1$ 
for the quantity $(p+1)|J_0(j_{2i-1})|^p$ to be decreasing in $p$ for all $i$, which holds from $p=6$ onwards.

Second, combining this last observation with \cref{prop:lim}, we will be done if we show that 
\[\sum_{i=1}^{\infty}i|J_0(j_{2i-1})|^6<\frac{2}{21\pi^2}.\]
From \eqref{eq:watson}, we have $j_{2i-1} \geq (2i-1)\pi$. Coupling this with the fact that Szeg{\H o}'s bound \eqref{eq:szegop.167} is decreasing in $x$, we have
\[J_0(j_{2i-1}) \leq \sqrt{\frac{2}{\pi j_{2i-1}}} \leq \sqrt{\frac{2}{\pi^2(2i-1)}},\] and so
$$
\sum_{i=1}^{\infty}i|J_0(j_{2i-1})|^6 \leq \frac{8}{\pi^6}\sum_{i=1}^\infty \frac{i}{(2i-1)^{3}}.
$$

Using partial fraction decomposition and rewriting using the Hurwitz zeta function, we get 
\begin{equation} \label{eq:hur1}
    \frac{8}{\pi^6}\sum_{i=1}^\infty \frac{i}{(2i-1)^{3}} = \frac{1}{\pi^6}\Big(\zeta\Big(2,-\frac{1}{2}\Big)-4+\frac{1}{2}
    \Big(\zeta\Big(3,-\frac{1}{2}\Big)-8\Big)\Big).
\end{equation}
Using the identities $\zeta\qty(s,\frac{1}{2})=\zeta\qty(s,-\frac{1}{2})-2^{s}$ for $s>1$ and $\zeta\qty(s,\frac{1}{2})=(2^s-1)\zeta(s)$, where $\zeta(s)$ is the usual Riemann zeta function, we obtain that \eqref{eq:hur1} is equal to
\begin{align} 
\frac{1}{\pi^6}\Big(\zeta\Big(2,\frac{1}{2}\Big)+\frac{1}{2}\zeta\Big(3,\frac{1}{2}\Big)\Big) 
= \frac{1}{\pi^6}\Big(3\zeta(2)+\frac{7}{2}\zeta(3)\Big). \label{eq:hur2}
\end{align}
Since $\zeta(2)=\pi^2/6$ and $\zeta(3)$ is Ap\'ery's constant which is $<1.2021$, \eqref{eq:hur2} is
$$
<\frac{1}{2\pi^6}\qty(\pi^2 + 7\cdot 1.2021) < 0.00951.
$$
On the other hand,
\[\frac{2}{21\pi^2} > 0.00964>0.00951>\sum_{i=1}^{\infty}i|J_0(j_{2i-1})|^6,\]
which concludes the proof.

We note that the result of \cref{thm:main} has a clear corollary on $\R P^2$,
\begin{corollary}\label{cor:rp}
There is a sequence of eigenfunctions $\widetilde{\psi}_n$ of the Laplacian on the real projective plane with its usual metric whose eigenvalues $\lambda_n\to\infty$ as $n\to\infty$ and such that for all $p\geq 6$,
\[\lim_{n\to\infty}\frac{\|\widetilde{\psi}_{n,+}\|_p}{\|\widetilde{\psi}_{n,-}\|_p} > 1.\]
\end{corollary}
\subsection*{Proof of \cref{cor:rp}}
Consider the sequence of eigenfunctions $\psi_n$ from \cref{thm:main}. We know they are even and so they descend to functions $\widetilde{\psi}_n$ on $\R P^2$ under the quotient of $S^2$ by the equivalence relation $x \sim -x$. Since this quotient is a local isometry, the functions $\widetilde{\psi}_n$ are eigenfunctions with the same eigenvalues. By lifting back up to the orientable double cover, we see that these eigenfunctions have the same ratio of positive to negative $L^p$ norms as for the sphere. Then by \cref{thm:main} the result follows.

\section{Conclusion.} The estimates made throughout our lemmas are fairly crude and the statement of our result for $p\geq 6$ was chosen for the niceness of the number; with effort, this value of 6 may be brought down. However, the restriction of $p>4$ from Lemma 6 seems to be a strict bound for approximations similar to the ones from our proof. Generalizations of the arguments may also be possible to higher dimensional spheres by studying the Gegenbauer polynomials.

This paper, along with the one by the authors and Eagles and Wang, establish a failure of generalized symmetry in model spaces of both zero and constant positive curvature. In the opinion of the authors, it is an interesting (and likely more challenging) question to investigate the conjecture 
on manifolds of constant negative curvature, where it is believed to hold due to the conjectures of quantum chaos alluded to in the Introduction.

\medskip

\subsection*{Acknowledgments}
\noindent
This research was conducted as part of the 2021 Fields Undergraduate Summer
Research Program. The authors are grateful to the Fields Institute for their financial support and facilitating our online collaboration. The authors would also like to thank \'Angel D. Mart\'inez and Francisco Torres de Lizaur for suggesting this project and reviewing an earlier version of this work.

\bibliographystyle{acm}
\bibliography{references.bib}

\begin{thebibliography}{10}

\bibitem{FUSRP1}
{\sc Beiner, G., Eagles, N.~M., Verreault, W., and Wang, R.}
\newblock A counterexample to symmetry of {$L^p$} norms of eigenfunctions.
\newblock {\em arXiv preprint arXiv:2208.14880\/} (2022).

\bibitem{Berry_Conjec}
{\sc Berry, M.~V.}
\newblock Regular and irregular semiclassical wavefunctions.
\newblock {\em J. Phys. A 10}, 12 (1977), 2083--2091.

\bibitem{cooper_1950}
{\sc Cooper, R.}
\newblock The extremal values of {L}egendre polynomials and of certain related
  functions.
\newblock {\em Proc. Cambridge Philos. Soc. 46\/} (1950), 549--554.

\bibitem{MR943927}
{\sc Donnelly, H., and Fefferman, C.}
\newblock Nodal sets of eigenfunctions on {R}iemannian manifolds.
\newblock {\em Invent. Math. 93}, 1 (1988), 161--183.

\bibitem{J-N}
{\sc Jakobson, D., and Nadirashvili, N.}
\newblock Quasi-symmetry of {$L^p$} norms of eigenfunctions.
\newblock {\em Comm. Anal. Geom. 10}, 2 (2002), 397--408.

\bibitem{jakobson2001geometric}
{\sc Jakobson, D., Nadirashvili, N., and Toth, J.}
\newblock Geometric properties of eigenfunctions.
\newblock {\em Russian Mathematical Surveys 56}, 6 (2001), 1085.

\bibitem{Marklof}
{\sc Marklof, J.}
\newblock The {B}erry-{T}abor conjecture.
\newblock In {\em European {C}ongress of {M}athematics, {V}ol. {II}
  ({B}arcelona, 2000)}, vol.~202 of {\em Progr. Math.} Birkh\"{a}user, Basel,
  2001, pp.~421--427.

\bibitem{eigentorus}
{\sc Mart\'{\i}nez, A.~D., and Torres~de Lizaur, F.}
\newblock Distribution symmetry of toral eigenfunctions.
\newblock {\em Rev. Mat. Iberoam. 38}, 4 (2022), 1371--1382.

\bibitem{signlegendre}
{\sc Mart\'inez, A.~D., and Torres~de Lizaur, F.}
\newblock Sign equidistribution of {L}egendre polynomials.
\newblock {\em {}\/} (Submitted for publication).

\bibitem{szego_bounds}
{\sc Szeg\H{o}, G.}
\newblock Inequalities for the zeros of {L}egendre polynomials and related
  functions.
\newblock {\em Trans. Amer. Math. Soc. 39}, 1 (1936), 1--17.

\bibitem{szego_orth_pol}
{\sc Szeg\H{o}, G.}
\newblock {\em Orthogonal polynomials}, fourth~ed.
\newblock American Mathematical Society Colloquium Publications, Vol. XXIII.
  American Mathematical Society, Providence, R.I., 1975.

\bibitem{pollard}
{\sc Tenenbaum, M., and Pollard, H.}
\newblock {\em Ordinary Differential Equations}.
\newblock Dover Books on Mathematics. Dover Publications, New York, 1985.
\newblock Reprint of the first (1963) edition.

\bibitem{watson_bessel}
{\sc Watson, G.~N.}
\newblock {\em A treatise on the theory of {B}essel functions}.
\newblock Cambridge Mathematical Library. Cambridge University Press,
  Cambridge, 1966.
\newblock Reprint of the second (1944) edition.

\end{thebibliography}
\parindent0pt
\end{document}